\documentclass[11pt, reqno]{amsart}
\usepackage{ amsmath, amsthm, amscd, amsfonts, amssymb, graphicx, color, enumerate}
\usepackage[bookmarksnumbered, colorlinks, plainpages]{hyperref}

\definecolor{TITLE}{rgb}{0.0,0.0,1.0}
\definecolor{AUTHOR1}{rgb}{0.00,0.59,0.00}
\definecolor{AUTHOR2}{rgb}{0.50,0.00,1.00}
\definecolor{SECTION}{rgb}{0.50,0.00,1.00}
\definecolor{FOOTTITLE}{rgb}{0.00,0.50,0.75}
\definecolor{THM}{rgb}{0.7,0.3,0.3}
\definecolor{SEC}{rgb}{0.6,0.1,.5}

\makeatletter
\@namedef{subjclassname@2010}{%
\textup{2010} Mathematics Subject Classification}
\makeatother
\newtheorem{theorem}{{\color{THM} Theorem}}[section]

\newtheorem{lemma}[theorem]{{\color{THM}Lemma}}
\newtheorem{proposition}[theorem]{{\color{THM}Proposition}}

\theoremstyle{definition}

\newtheorem{example}[theorem]{{\color{THM}Example}}

\newtheorem{remark}[theorem]{{\color{THM}Remark}}
\numberwithin{equation}{section}

\numberwithin{equation}{section}

\def\speaker{$ ^{*} $\protect\footnotetext{$ ^{*} $\lowercase{Corresponding Author.}}}
\begin{document}

\title {the Radon Transform on Function Spaces Related to Homogenous Spaces }

\author{ T. Derikvand, R. A. Kamyabi-Gol\speaker$   $ and M. Janfada}

\address{International Campus, Faculty of Mathematic Sciences, Ferdowsi University of Mashhad, Mashhad, Iran}
\email{derikvand@miau.ac.ir}

\address{ Department of Pure Mathematics and  Centre of Excellence
in Analysis on Algebraic Structures (CEAAS), Ferdowsi University of
Mashhad, P.O. Box 1159, Mashhad 91775, Iran}
\email{kamyabi@um.ac.ir}

\address{ Department of Pure Mathematics, Ferdowsi University of
Mashhad, P.O. Box 1159, Mashhad 91775, Iran}
\email{Janfada@um.ac.ir}

\begin{abstract}
We are going to study some conditions on which the Radon transform and its dual are invertible. Two function spaces are introduced that the Radon transform on which is bijective linear operator. In this regards, a reconstruction formula is constructed. Finally the results are supported by an example. \end{abstract}


\maketitle
\section{INTRODUCTION}
Let $H$ and $K$ be two closed subgroups of a locally compact group $G$. Consider two $ G $-spaces $G/K$ and $G/H$. Recall that the terminology " $X$ is a $G$-space" means that $ X $ is a locally compact Hausdorff space on which the topological group $G$ acts by an action map transitively and continuously. S. Helgason in 1966 studied the Radon transform in a general framework of $ G- $spaces for a topological group $ G $. Therein Radon transform $ R_{K, H}:C_{c}(G/K) \rightarrow C(G/H)$  and its dual $R^{*}_{K, H}:C_{c}(G/H) \rightarrow C(G/K)$
were defined by:
\begin{align}
R_{K, H}f(xH)=\int_{H/L}{f(xhK)d(hL)}\qquad (f\in C_{c}(G/K)),\label{eq11}
\end{align}
and
\begin{align}
R^{*}_{K, H}\varphi(xK)=\int_{K/L}{\varphi(xkH)d(kL)}\qquad (\varphi \in C_{c}(G/H)).\label{eq12}
\end{align} 
Where $L:=H \cap K$ and also $d\sigma$=$d(kL)$ and $d\eta=d(hL)$ are two $K-$invariant and $H-$invariant Radon measures on $K/L$ and $H/L$, respectively (for more detail see \cite{4}).
Considering Proposition $ 3.7. $ in \cite{2}, $G/K$ and $G/H$ are $ G-$homeomorphic if and only if $H$ and $K$ are conjugate in $ G $. One may ask the question: What are the necessary and sufficient conditions on $ K $ and $ H $ to exist an isomorphism between $ C_{c}(G/K)) $ and $ C_{c}(G/H)) $?
We deal with the following principal problems that were presented by S. Helgason in \cite{4}.\\
\textbf{A.} Relate function spaces on $X:=G/K$ and on $Y:=G/Y$ by means of the transforms $f \mapsto R_{K, H}f$, $ \varphi \mapsto  R^{*}_{K, H}\varphi$. In particular, determine their ranges and kernels.\\
\textbf{B.} Invert the transforms $f \mapsto R_{K, H}f $ and $ \varphi \mapsto R^{*}_{K, H}\varphi$ on suitable function spaces.\\
Considering the Helgason questions in the special case, where $L=K$ is an arbitrary closed subgroup of $H$, we get
\begin{align}
R_{L, H}f(gH)=\int_{H/L}{f(ghL)d(hL)} \qquad (f \in C_{c}(G/L)),\label{eq13}
\end{align}
and
\begin{align}
R^{*}\varphi(gL)=\varphi(gH)\int_{L/L}{d(L)}.\qquad (\varphi \in C_{c}(G/H)).\label{eq14}
\end{align}
Now we are interested in answering the following questions regarding special case of Helgason's questions.
\begin{enumerate}
\item [(i)] Is $R_{L, H}$ a surjective bounded operator from $C_{c}(G/L)$ to $C_{c}(G/H)$?
\item [(ii)] What are the changes in the definition of Radon transform and its properties if $G/H$ and $H/L$ do not possess respectively $G-$invariant and $H-$invariant Radon measures?
\item [(iii)] Under which conditions on $H$ and $L$ is the Radon transform injective?
\end{enumerate}
Also, we are interested in finding function spaces that could answer the above questions regarding as a special case of Helgason's problems.\\
The outline of the rest of this paper is as follows: In section 2 we study the preliminaries including a brief summary on homogeneous spaces and the (quasi) invariant measures on them. 
The third section is devoted to introducing two suitable function spaces on which the Radon transform and its dual are invertible. 
\section{PRELIMINARIES}
In the sequel, $H$ is a closed subgroup of a locally compact group $G$ and $ dx $, $ dh $ are the left Haar measures on $ G $ and $ H $, respectively.
We recall that the modular function $ \triangle_{G} $ is a continuous homomorphism from $G$ into the multiplicative group $ \mathbb{R}^{+} $. Furthermore,
\[\int_{G} f(y)dy=\triangle_{G}(x)\int_{G} f(yx)dy\]
where $ f\in C_{c}(G) $, the space of continuous functions on $ G $ with compact support, and $ x\in G $.
A locally compact group $ G $ is called unimodular if $ \triangle_{G}(x)=1 $, for all $ x \in G $. A compact group $ G $ is always unimodular.
It is known that $C_c(G/H)$ consists of all $P_{H}f$ functions, where $f\in C_c(G)$ and
\[P_{H}f(xH)=\int_{H}f(xh)dh\quad(x\in G).\]
Moreover, $P_{H}:C_c(G)\rightarrow C_c(G/H)$ is a bounded linear operator which is not injective (see subsection 2.6  of \cite{3}). Suppose that $ \mu $ is a Radon measure on $ G/H $. For all $ x\in G $ we
define the translation of $ \mu $ by $ x $, by $ \mu_{x}(E)=\mu (xE) $, where $ E $ is a Borel subset of $ G/H $.
Then $ \mu $ is said to be $ G-$invariant if $ \mu_{x}=\mu $, for all $ x\in G $, and is said to be strongly quasi-invariant, if there is a continuous function $ \lambda:G\times G/H\rightarrow (0, +\infty) $ which satisfies
\[d\mu_{x}(yH)=\lambda(x, yH)d\mu(yH).\]
If the function $ \lambda(x, .) $ reduces to a constant for each $ x \in G $, then $ \mu $ is called relatively invariant under $ G $. We consider a rho-function for the pair $ (G, H) $ as a continuous function $ \rho:G\rightarrow(0, +\infty) $ for which $\rho(xh)=\triangle_{H}(h)\triangle_{G}(h)^{-1}\rho(x)$, for each $x\in G$ and $h\in H$. It is well known that $ (G, H) $ admits a rho-function and for every rho-function $ \rho $ there is a strongly quasi-invariant measure $ \mu $ on $ G/H $ such that
\[\int_{G} f(x)\rho(x)dx=\int_{G/H}\int_{H} f(xh)dhd\mu(xH) \qquad (f\in C_{c}(G)).\]
This equation is called quotient integral formula . The measure $\mu$ also satisfies
\[\frac{d\mu_{x}}{d\mu}(yH)=\frac{\rho(xy)}{\rho(y)}\qquad (x,y \in G).\]
Every strongly quasi-invariant measure on $ G/H $ arises from a rho-function in this
manner, and all of these measures are strongly equivalent (Proposition 2.54 and Theorem 2.56 of \cite{3}). Therefore, if $ \mu $ is a strongly quasi-invariant measure on $ G/H $, then the measures $ \mu_{x} $, $ x\in G $, are all mutually absolutely continuous. Trivially, $ G/H $ has a $ G- $invariant Radon measure if and only if the constant function $ \rho(x)=1 $, $ x\in G $, is a rho-function for the pair $ (G, H) $.\\

If $ \mu $ is a strongly quasi invariant measure on $G/H$ which is associated with the rho$-$function $\rho$ for the pair $ (G, H) $, then the mapping $T_{H}: L^{1}(G)\rightarrow L^{1}(G/H)$ defined almost everywhere by
\[T_{H}f(xH)=\int_{H}\frac{f(xh)}{\rho(xh)}dh \qquad (f\in L^{1}(G))\]
is a surjective bounded linear operator with $\parallel T_{H} \parallel \leq1$ (see \cite{3}) and also satisfies the generalized Mackey-Bruhat formula,
\begin{align}
\int_{G}f(x)dx=\int_{G/H}T_{H}f(xH)d\mu(xH)\qquad (f\in L^{1}(G)),\label{eq21}
\end{align}
which is also known as the quotient integral formula.\\
\section{invertibility and reconstruction formula}
In this section, two suitable function spaces will be presented on which the Radon transform and its dual are invertible.\\
We assume that $G$ is a locally compact group and $H$
is its closed subgroup. The quotient space $G/H$ is
considered as a homogeneous space that $G$ acts on it from the
left and $\pi_{H}:G\longrightarrow G/H$ denotes the canonical map. It is well known that $\pi_{H}$ is open, surjective and continuous. This can be generalized to homogeneous spaces as follows:
\begin{proposition}\label{pr31}
Let $H$  be a closed subgroup of locally compact group $G$,  and $L$ be a closed subgroup of $H$. The map $\pi_{L, H}:G/L\to G/H$ defined by $\pi_{L, H}(xL)=xH$ is open, surjective and continuous.
\end{proposition}
\begin{proof}
Trivially the map $\pi_{L, H}$ is well defined.
We have $\pi_{H}=\pi_{L, H}\circ \pi_{L}$, where $\pi_{L}:G\to G/L$, $\pi_{H}:G\to G/H$ are the canonical quotient maps. By using Lemma $ 3.1.9 $ of \cite{7}, since $\pi_{L}$ is an open and surjective map and $\pi_{H}$ is a continuous map, we have $\pi_{L, H}$ is continuous. Also, $ \pi_{L, H}$ is an open map because for each open set $U$ in $G/L$
\begin{align*}
\pi_{L, H}(U)&=\{\pi_{L, H}(xL)\ |\ xL\in U\}\\
&=\{xH\ |\ xL\in U\}\\
&=\{\pi_{H}(x)\ |\ xL\in U\}\\
&=\pi_{H}\{x\in G\ |\ xL\in U\}\\
&=\pi_{H}\{x\in G\ |\ \pi_{L}(x)\in U\}\\
&=\pi_{H}(\pi_{L}^{-1}(U)),
\end{align*}
meanwhile $\pi_{L}$ is a continuous map and $\pi_{H}$ is an open map.\\
\end{proof}
It is worthwhile to note that if $G/L$ is locally compact (respectively, compact), then $G/H$ is also locally compact (respectively, compact).
Now suppose that $\Delta_H\big|_L=\Delta_L$, where $\Delta_H$  and  $\Delta_L$ are the modular functions of $H$ and $L$, respectively. In this case, the quotient integral formula \eqref{eq21} guarantees the existence of a unique (up to a constant) $H$-invariant measure $\eta$ on $H/L$. Now we define the Radon transform $R_{L, H}:C_c(G/L)\to C_c(G/H)$ by
\begin{align}
R_{L, H}f(xH)=\int_{H/L}f(xhL)d\eta(hL)\qquad (f\in C_{c}(G/L)),\label{eq31}
\end{align}
which is well defined by the following proposition. Note that if $L $ is the trivial subgroup of $ H $ then clearly $ R_{L, H}$  is $P_{H}$.
\begin{proposition}\label{pr32}
If $f\in C_c(G/L)$, then $ R_{L, H}f \in C_c(G/H) $ and $R_{L, H}[(\phi\circ \pi_{L, H})\cdot f]=\phi\cdot R_{L, H}f$, for all $ \phi \in C_c(G/H) $.
\end{proposition}
\begin{proof}
The map $ R_{L, H}:C_c(G/L) \to C_c(G/H) $ is well defined by the left-invariance of $\eta$ and the fact that any $ f \in C_c(G/L)$ is uniformly continuous (Lemma 1.3.3 in \cite{6}). Fix $x_0\in G$ and pick a compact neighborhood $U$ of $x_{0}$ in $G$, then for every $x\in U$, the function $hL\mapsto f(xhL)$ is supported in the compact set $(U^{-1}suppf)\cap H/L$. Indeed we have
\begin{align*}
\overline{\{hL: f(xhL)\neq0\}}\subseteq (U^{-1}suppf)\cap H/L.
\end{align*}
$M:=(U^{-1}suppf)\cap H/L$ and $\varepsilon>0$ are given. By the left uniformly continuity of $f$, there exists a neighborhood $V_{1}$ of $1$ in $G$ such that
\[|f(xhL)-f(x_{0}hL)|<\frac{\varepsilon}{\eta(M)},\]
whenever $x,x_{0}\in G$ with $xx_{0}^{-1}\in V_{1}$. Put $V_{2}:=V_{1}x_{0}^{-1}$, so
$|f(xhL)-f(x_{0}hL)|<\frac{\varepsilon}{\eta(M)}$ for every $x\in V_{2}.$
Now for every $x\in V_{3}:= V_{2}\cap U$, we have
\begin{align*}
\vert R_{L, H}f(xH)-R_{L, H}f(x_{0}H)\vert &=\vert \int_{H/L}(f(xhL)-f(x_{0}hL))d\eta(hL)\vert\\
&\leq\int_{H/L}\vert f(xhL)-f(x_{0}hL)\vert d\eta(hL)\\
&\leq \int_{M}\vert f(xhL)-f(x_{0}hL)\vert d\eta(hL)\\
&<\int_{M}\frac{\varepsilon}{\eta(M)}d\eta(hL)\\
&<\varepsilon,
\end{align*}
which shows the continuity of $R_{L, H}f$.\\
Moreover, if $\pi_{L, H}(xL)\not\in \pi_{L, H}(supp\ f)$, then $ \pi_{L, H}(xhL)\not\in \pi_{L, H}(supp\ f)$ for all $h\in H$. This implies that $ xhL\not\in supp\ f $, for all $h\in H$; so $ f(xhL)=0 $ for all $h\in H$. Hence $ \int_{H/L}f(xhL)d\eta(hL)=R_{L, H}f(xH)=0 $, so $ xH\not\in supp(R_{L, H}f) $. Thus $(supp\ R_{L, H}f) \subseteq \pi_{L, H}(supp\ f) $ and so is compact.\\
Finally if $\phi\in C(G/H)$, we have
\begin{align*}
R_{L, H}[(\phi\circ \pi_{L, H})\cdot f](xH)&=\int_{H/L}[\phi\circ \pi_{L, H})\cdot f](xhL)d\eta (hL)\\
&=\int_{H/L}(\phi\circ \pi_{L, H})(xhL)\cdot f(xhL)d\eta(hL)\\
&=\int_{H/L} \phi(xH)\cdot f(xhL)d\eta (hL)\\
&=\phi(xH)\int_{H/L}f(xhL)d\eta(hL)\\
&=\phi(xH)R_{L, H}f(xH)\\&=(\phi\cdot R_{L, H}f)(xH),
\end{align*}
so $R_{L, H}[(\phi\circ \pi_{L, H})\cdot f]=\phi\cdot R_{L, H}f$.\\
\end{proof}
Let $H$ be a compact subgroup of a locally compact group $G$ with left-invariant Haar measure $dh$, and $ L $  is a closed subgroup of $ H $. Then by compactness of $ H $ there are two invariant Radon measures $ d(xL) $, $ d(hL) $  on $G/L$ and $H/L$, respectively.  By using \eqref{eq11}, we get\\
\begin{align}
R_{L, H}f(xH)=\int_{H/L}{f(xhL)d\eta(hL)} \qquad (f \in C_{c}(G/L)). \label{eq41}
\end{align}
Also, for all $\varphi$ in $C_{c}(G/H)$ we have\\
\begin{align*}
R^{*}_{L, H}\varphi (xL)\
&=\int_{L/L}\varphi(xlH){d\sigma(L)}\\
&=\varphi(xH)\int_{{L}}{d\sigma(L)}\\
&=\varphi(xH)\\
&=(\varphi \circ \pi_{L, H})(xL),
\end{align*}
where $\pi_{L, H}:G/L\to G/H$ given by $\pi_{L, H}(xL)=xH$, is the canonical map. From now on, we abbreviate $ \varphi \circ \pi_{L, H} $ by $ \varphi_{\pi_{L, H}} $.
Now for all $f$ in $C_{c}(G/L)$, we have\\
\begin{align*}
[(R^{*}_{L, H} \circ R_{L, H})f](xL)\
&=[R^{*}_{L, H} (R_{L, H}f)](xL)\\
&=[(R_{L, H}f) \circ \pi_{L, H}](xL),
\end{align*}
for all $xL \in G/L$. Thus $ (R^{*}_{L, H} \circ R_{L, H})f=R_{L, H}f \circ \pi_{L, H} =(R_{L, H}f)_{\pi_{L, H}}$. For any $\varphi $ in $C_{c}(G/H)$, we have\\
\begin{align*}
[(R_{L, H} \circ R^{*}_{L, H})\varphi](xH)\
&=[R_{L, H} (R^{*}_{L, H}\varphi)](xH)\\
&=\int_{H/L}R^{*}_{L, H}\varphi(xhL){d\eta(hL)}\\
&=\int_{H/L}\varphi\circ\pi_{L, H}(xhL){d\eta(hL)}\\
&=\int_{H/L}\varphi[\pi_{L, H}(xhL)]{d\eta(hL)}\\
&=\int_{H/L}\varphi(xH){d\eta(hL)}\\
&=\varphi(xH)\int_{H/L}{d\eta(hL)}\\
&=\varphi(xH),
\end{align*}
for all $xH \in G/H$.
Note that if $L$ is a closed subgroup of a compact group $H$, then $H/L$ is compact.
Thus $ R_{L, H}\circ R^{*}_{L, H}=id_{C_{c}(G/H)} $.\\
\begin{proposition}
Let $H$ be a compact subgroup of a locally compact group $G$ and suppose $ L $  is a closed subgroup of $ H $.\\
a) There exists a unique Radon measure $ \mu $ such that
\begin{center}
$ (P^{*}_{H} \circ P_{H})f=f\ast\mu \quad(f\in C_c(G))$
\end{center}
b) $ P_{H}\circ P^{*}_{H}=id_{C_{c}(G/H)} $.
\end{proposition}
\begin{proof}
a) For a given $ f $ in $ C_c(G) $, define $ F: C_c(G)\rightarrow \mathbb{C}$ by $ F(f):=\int_{H}f\mid_{H}(h)dh\
 $, where $ dh $ is the Haar measure of $ H $. Note that the integral is well defined since $f\mid_{H}$  $\in C_{c}(H) $.  Also trivially, $ F $ is a positive linear functional on $C_{c}(G)$. So it defines
a Radon measure $ \mu $ on $ G $ such that $ F(f)=\int_{G}fd\mu $. We obtain
\begin{align}
f\ast\mu(x)=\int_{G}f(xy)d\mu(y)=\int_{H}f(xh)dh \quad(f\in C_c(G)) \label{eq42}
\end{align}
and
\begin{align}
[(P^{*}_{H} \circ P_{H})f](x)=P^{*}_{H} f(xH)=\int_{H}f(xh)dh. \label{eq43}
\end{align}
Now the identities \eqref{eq42}, \eqref{eq43} imply the result.\\
b) Let $\varphi $ be in $C_{c}(G/H)$. Then we have\\
\begin{align*}
(P_{H} \circ P^{*}_{H})\varphi(xH)\
&=P_{H} (P^{*}_{H}\varphi)(xH)\\
&=\int_{H}P^{*}_{H}\varphi(xh){dh}\\
&=\int_{H}\varphi\circ\pi_{H}(xh){dh}\\
&=\int_{H}\varphi[\pi_{H}(xh)]{dh}\\
&=\int_{H}\varphi(xH){dh}\\
&=\varphi(xH)\int_{H}{dh}\\
&=\varphi(xH),
\end{align*}
for all $xH \in G/H$. Thus $ P_{H}\circ P^{*}_{H}=id_{C_{c}(G/H)} $.\\
\end{proof}

Now let $H$ be a compact subgroup of a locally compact group $G$ and $ L $  be a closed subgroup of $ H $. Suppose that
\begin{center}
$ C_c(G/L  :  H):=\lbrace f \in C_c(G/L) : f(xhL)=f(xL) $,   $\forall x \in G,$   $ \forall h \in H\rbrace$.\\
\end{center}
One can easily prove that
\begin{center}
$ C_c(G/L  :  H)=\lbrace \varphi_{\pi_{L, H}}:=\varphi \circ \pi_{L, H} : \varphi \in C_c(G/H) \rbrace$.\\
\end{center}
If $ f \in C_c(G/L) $ and $ g \in C_c(G/L  :  H) $ then $ f*g \in C_c(G/L  :  H) $ and therefore $ C_c(G/L  :  H) $ is a left ideal
and also a sub-algebra of $ C_c(G/L) $. The algebra $ C_c(G/L  :  H)$ is unital and the unit element is $ \chi_{H/L} $. If $ K $ is another compact subgroup of $ G $ then $ C_c(G/L  :  H)=C_c(G/L  :  K) $ if and only if $ H=K $.
\begin{theorem}\label{th2}
Let $H$ be a compact subgroup of a locally compact group $G$ with left Haar measure $dh$, and $L$ be a closed subgroup of $H$ with left Haar measure $dl$. Suppose $d\sigma$=$d(xL)$ and $d\eta=d(xH)$ are two $L-$invariant and $H-$invariant Radon measures on $G/L$ and $H/L$, respectively. Then the Radon transform $ R_{L, H} $ is a bijective linear operator from $C_c(G/L  :  H)  $ to $ C_c(G/H) $; Further, we have
$\left( R_{L, H}f\right)_{\pi_{L, H}}=f$, for all $ f \in C_c(G/L  :  H) $.  
\end{theorem}
\begin{proof}
Let $\phi\in C_c(G/H)$. Since $\phi_{\pi_{L, H}} \in C_c(G/L)$, we get
\begin{align*}
\phi_{\pi_{L, H}}(xhL)\
&=\phi\circ \pi_{L, H}(xhL)\\
&=\phi(xH)\\
&=\phi\circ \pi_{L, H}(xL)\\
&=\phi_{\pi_{L, H}}(xL),
\end{align*}
for all $ x \in G $ and $ h \in H $. Thus $ \phi_{\pi_{L, H}} \in C_c(G/L  :  H) $ and for all $ x \in G $ we have
\begin{align*}
R_{L, H}\phi_{\pi_{L, H}}(xH)\
&=\int_{H/L}\phi_{\pi_{L, H}}(xhL)d(hL)\\
&=\int_{H/L}\phi\circ \pi_{L, H}(xhL)d(hL)\\
&=\int_{H/L}\phi(xH)d(hL)\\
&=\phi(xH)\int_{H/L}d(hL)\\
&=\phi(xH).
\end{align*}
So $ R_{L, H} $ is surjective.\\
Let $ f \in C_c(G/L  :  H) $ and $ R_{L, H}f=0 $. Thus for all $ xL \in G/L $
\begin{align*}
f(xL)\
&=f(xL)\int_{H/L}d(hL)\\
&=\int_{H/L}f(xL)d(hL)\\
&=\int_{H/L}f(xhL)d(hL)\\
&=R_{L, H}f(xH)\\
&=0.
\end{align*}
Thus $ R_{L, H} $ is a one to one linear transform. Also $ f(xL)=R_{L, H}f(xH) $ implies that $\left( R_{L, H}f\right)_{\pi_{L, H}}=f$, for all $ f \in C_c(G/L  :  H) $.   
\end{proof}
\begin{example}
Let $ G $ denote the multiplicative group of nonzero complex numbers and $ L:=\lbrace\pm1\rbrace, H:=\lbrace\pm1, \pm{i}\rbrace $ be two compact subgroups of it. Assume that for any $ z $ in $\mathbb{C}-\lbrace 0 \rbrace $, $\varphi(z)= \begin{cases} 1-\mid z \mid &\mbox{for } \mid z \mid<1 \\
0 & \mbox{for } \mid z \mid \geq1 \end{cases}$. Then clearly $\varphi\in C_{c}(G)$. Now define $ f(zL):=P_{L}\varphi(zL)=2\varphi(z)$ for all $ z \in \mathbb{C}-\lbrace 0 \rbrace $. Then it is easy to verify that  $f(zL)=f(zhL)$. So $f \in C_c(G/L  :  H)$. Now we have
$\left( R_{L, H}f\right) _{\pi_{L, H}}(zL)=R_{L, H}f (zH) =2\varphi(z)=f(zL)$.
\end{example}
Now Let $H$ and $K$ be two compact subgroups of a locally compact group $G$ and $ L:=H\cap K $. Suppose that
\[ C_c(G/K  :  H):=\lbrace f \in C_c(G/K) : f(xhK)=f(xK) ,   \forall x \in G,    \forall h \in H\rbrace.\]
One can show that $ C_c(G/K  :  H) $ is a left ideal in the algebra $ C_c(G/K) $.
\begin{proposition}\label{pr3}
Let $H$ and $K$ be two compact subgroups of a locally compact group $G$ with left Haar measures $dh$ and $dk$, respectively. Also, let $L:=H\cap K$ with left Haar measure $dl$. Assume that $ d(kL) $, $ d(hL) $ are two invariant Radon measures on $K/L$ and $H/L$, respectively. Then the Radon transform $R_{K, H}:C_{c}(G/K : H)\rightarrow C(G/H : K)$  and its dual $R^{*}_{K, H}:C_{c}(G/H : K) \rightarrow C(G/K : H)$ are bijective linear operators. Furthermore, we have $(R_{K, H}f)_{\pi_{L, H}}=f_{\pi_{L, K}}$ for all $ f $ in $ C_{c}(G/K : H) $.\\
\end{proposition}
\begin{proof}
It is clear that $ R_{K, H}(C_{c}(G/K : H)) \subseteq C(G/H) $. Since $ R_{K, H}f(xH)=f(xK) $, $ R_{K, H}f $ belongs to $ C(G/H : K) $, for any $ f $ in $ C_{c}(G/K : H) $. Then $R_{K, H}:C_{c}(G/K : H)\rightarrow C(G/H : K)$ is well-defined and in a similar way $ R^{*}_{K, H} $ is well-defined too. 
For any $ \phi \in C(G/H : K)$, by using the Theorem \ref{th2}, there exists a function $\phi_{\pi_{L, H}}\in C_c(G/L : H)$ such that $ R_{L,H}\phi_{\pi_{L, H}}=\phi $. Put $ \psi:=R_{L,K}\phi_{\pi_{L, H}} $. We have
\begin{align*}
\psi(xhK)\
&=\int_{K/L}\phi_{\pi_{L, H}}(xhkL)d(kL)\\
&=\int_{K/L}\phi\circ\pi_{L, H}(xhkL)d(kL)\\
&=\int_{K/L}\phi(xhkH)d(kL)\\
&=\int_{K/L}\phi(xkH)d(kL)\\
&=\psi(xK).
\end{align*}
Thus $ \psi \in C_c(G/K  :  H) $ and for all $ x \in G $ we have
\begin{align*}
R_{K,H}\psi(xH)\
&=\int_{H/L}\psi(xhK)d(hL)\\
&=\int_{H/L}\int_{K/L}\phi(xkH)d(kL)d(hL)\\
&=\phi(xH).\\
\end{align*}
So $ R_{K,H} $ is surjective.\\
Let $ f \in C_c(G/K  :  H) $ and $ R_{K, H}f=0 $. Then for all $ xK \in G/K $
\begin{align*}
f(xK)\
&=f(xK)\int_{H/L}d(hL)\\
&=\int_{H/L}f(xK)d(hL)\\
&=\int_{H/L}f(xhK)d(hL)\\
&=R_{K,H}f(xH)\\
&=0.
\end{align*}
So $ f\equiv0 $. Thus $ R_{K,H} $ is a one to one linear transform. Also, for all $ f \in C_c(G/K  :  H) $ and $ xL \in G/L $, we have
\begin{align*}
(R_{K, H}f)_{\pi_{L, H}}(xL)\
&=R_{K, H}f\circ \pi_{L, H}(xL)\\
&=R_{K, H}f(xH)\\
&=\int_{H/L}f(xhK)d(hL)\\
&=\int_{H/L}f(xK)d(hL)\\
&=f(xK)\int_{H/L}d(hL)\\
&=f\circ \pi_{L, K}(xL)\\
&=f_{\pi_{L, K}}(xL).
\end{align*}
\end{proof}
From now on, Consider $G/K$ and $G/H$ as two G-spaces on which $K$ and $H$ are two compact subgroups of a locally compact group $G$. Suppose that
\[ C_c(G  :  K):=\lbrace f \in C_c(G) : f(xk)=f(x) ,   \forall x \in G,    \forall k \in K\rbrace.\]
One can show that $ C_c(G  :  K) $ and $ C_c(G  :  H) $ are two left ideals in the algebra $ C_c(G) $ and $ C_c(G  :  K)=\lbrace \varphi_{\pi_{K}}=\varphi \circ \pi_{K} : \varphi \in C_{c}(G/K) \rbrace $, where $ \pi_{K}$ is the canonical map, moreover $T_{ K}:C_{c}(G : K)\rightarrow C(G/K)$ given by
\[T_{K}f(xK)=\int_{K}f(xk)dk \qquad (f\in C_{c}(G))\]
 is a bijective map (see \cite{1}). It is known that two G-spaces $G/K$ and $G/H$ are homeomorphic if and only if $K$ and $H$ are conjugate in $ G $ (see [4, Proposition 3.7]). In what follows, we shall prove that the algebras $ C_c(G/K) $ and $ C_c(G/H) $ are isomrphic if $G/K$ and $G/H$ are homeomorphic as G-spaces. 
\begin{lemma}\label{le4}
Let $H$ and $K$ be two compact subgroups of a locally compact group $G$. If the homogeneous spaces $G/K$ and $G/H$ are homeomorphic then $ C_c(G  :  K) $ and $ C_c(G  :  H) $ are isomorphic as the subalgebras of $ C_c(G) $.
\end{lemma}
\begin{proof}
Since $G/K$ and $G/H$ are homeomorphic so the subgroups $K$ and $H$ are conjugate in $ G $ and  there exists $ g_{0}\in G $ such that  the map $\theta : G/K \rightarrow G/H$ defined by  $gK \mapsto gg_{0}H $ is a homeomorphism. Now define the map  $ \tau : C_{c}(G : K) \rightarrow C_{c}(G : H)$ given by $\varphi_{\pi_{K}} \mapsto \varphi\circ\theta^{-1}_{\pi_{H}} $, where $ \pi_{H} $  is the canonical map and $ \theta^{-1} $ defined by  $\theta^{-1}(gH)=gg_{0}^{-1}K $ is the inverse of the homeomorphism $ \theta $. Since $ \pi_{H} $, $ \theta^{-1} $ and $ \varphi $ are continuous, we can conclude that $ \tau(\varphi_{\pi_{K}}) $ is continuous, also 
\[ \tau (\varphi_{\pi_{K}})(xh)=\varphi \circ \theta^{-1} \circ \pi_{H}(xh)=\varphi \circ \theta^{-1}(xhH)=\tau (\varphi_{\pi_{K}})(x).\]
Moreover, we have $ supp( \tau(\varphi_{\pi_{K}}))\subseteq supp (\varphi)$ so it is compact. Therefore the map $ \tau $ is well-defined.

To prove that $ \tau $ is surjective, let $ \psi_{\pi_{H}} $ be an arbitrary element in $C_{c}(G : H)$ and put $ \varphi:=\psi\circ \theta $. Then $ \varphi_{\pi_{K}}=\varphi\circ\pi_{K}\in C_{c}(G : K)$ and we have
\[ \tau(\varphi_{\pi_{K}})=\varphi \circ \theta^{-1} \circ \pi_{H}=\psi \circ \theta \circ \theta^{-1} \circ \pi_{H}=\psi \circ \pi_{H}=\psi_{\pi_{H}}.\]
Proving injectivity: \\
Suppose that $\tau(\varphi_{\pi_{K}})=0$ for any $ \varphi_{\pi_{K}}\in C_{c}(G : K) $ then
\begin{align*}
0\
&=\tau(\varphi_{\pi_{K}})(x)\\
&=\varphi \circ \theta^{-1} \circ \pi_{H}(x)\\
&=\varphi (\theta^{-1}(xH))\\
&=\varphi (xg_{0}^{-1}K),
\end{align*}
for any $ x\in G $. By replacing $x$ with $xg_{0}$, we get $\varphi (xK)=0$. So $ \varphi_{\pi_{K}(x)=0} $, for all $ x\in G $. This means that $ \tau $ is injective.   
\begin{remark}
It is straightforward to check that the linear oprator $ \tau $ is norm-decreasing. 
\end{remark}
\end{proof}
\begin{theorem}\label{pr5}
Let $H$ and $K$ be two compact subgroups of a locally compact group $G$. If the homogeneous spaces $G/K$ and $G/H$ are homeomorphic then two algebras $C_c(G/K) $ and $ C_c(G/H) $ are isomorphic.
\end{theorem}
\begin{proof}
By using Proposition 3.1 in \cite{1}, $T_{ H}$ and $T_{ K}$  are bijectives. Consider the isomorphism $ \tau : C_{c}(G : K) \rightarrow C_{c}(G : H)$ according to the Lamma \ref{le4} and define $ T : C_{c}(G/K) \rightarrow C_{c}(G/H)$ by $T:=T_{ H} \circ \tau \circ T_{ K}^{-1}. $ Assume that $ \varphi \in C_{c}(G/K) $. Then, $ T \varphi (xH)=\varphi (xg_{0}^{-1}K )$. It is easy to see that for any $ \varphi_{1}=\varphi_{2} \in C_{c}(G/K) $ we get $  T \varphi_{1}=T \varphi_{2}  $. So the map $ T $ is well- defined. To see that the map is injective, Let $ T \varphi=\varphi (xg_{0}^{-1}K) =0$, for all $ x\in G $. Replace $ x $ by $ xg_{0} $. Then the injectivity is trivial. It remains only to show that it is surjective. Let $ \psi $ be an arbitrary element in $C_{c}(G/H)$. Then $ \psi{\pi_{H}}\in C_{c}(G : K) $ and $ T (T_{K} \circ  \tau^{-1} \circ \psi{\pi_{H}})=T_{H} \circ  \tau \circ T_{K}^{-1} \circ T_{K} \circ  \tau^{-1} \circ \psi_{\pi_{H}}=T_{H} \psi_{\pi_{H}}=\psi $. So the result was obtained. 
  
\end{proof}
For the inverse of this theorem, the well-known results of J. G. Wendel have been generalized by the authores.

\end{document}